\documentclass{MyStyle}
\usepackage{bbm}

\usepackage{latexsym,amssymb,amsmath}
\usepackage[latin1]{inputenc}
\usepackage[T1]{fontenc}
\usepackage{graphicx}
\usepackage{multirow}
\usepackage{hyperref}

\newcommand{\nbR}{\mathbb{R}}
\newcommand{\nbN}{\mathbb{N}}

\newcommand{\nbC}{\mathbb{C}}
\newcommand{\nbu}{\mathbb{I}}

\author{Christian Paroissin\addressmark{1} and Landy Rabehasaina\addressmark{2}} 
\title[Passage times of perturbed subordinators with application to reliability]{Passage times of perturbed subordinators with application to reliability}

\address{\addressmark{1}Universit\'e de Pau et des Pays de l'Adour, Laboratoire de Math\'ematiques et de leurs Applications - UMR CNRS 5142, Avenue de l'Universit\'e, 64013 Pau cedex, France, {\tt cparoiss@univ-pau.fr}  \\
\addressmark{2} Universit\'e de de Franche-Comté, Laboratoire de Math\'ematiques - UMR CNRS 6623, 16 Route de Gray, 25030 Besan\c{c}on cedex, France, {\tt lrabehas@univ-fcomte.fr}}



\newtheorem{thm}{\noindent Theorem}[section]
\newtheorem{lemm}{\noindent Lemma}[section]

\newtheorem{prop}{\noindent Proposition}[section]

\newtheorem{remark}[theorem]{Remark}

\begin{document}
\maketitle

\begin{abstract}
We consider a wide class of increasing Lévy processes perturbed by an independent Brownian motion as
a degradation model. Such family contains almost all classical degradation models considered in the
literature. Classically failure time associated to such model is defined as the hitting time or
the first-passage time of a fixed level. Since sample paths are not in general increasing, we
consider also the last-passage time as the failure time following a recent work by Barker and Newby
\cite{Barker_Newby}. We address here the problem of determining the distribution of the
first-passage time and of the last-passage time. In the last section we consider a maintenance
policy for such models.
\end{abstract}


\section{Introduction and Model}

For several decades, degradation data have been more and more used to understand ageing of a
device, instead of only failure data. The most widely used stochastic processes for degradation
models belong to the class of L\'evy processes. More precisely, the three main models are the
following ones: (a) Brownian motion with (positive) drift; (b) gamma processes; (c) compound
Poisson processes. More generally we consider a broad class of L\'evy processes corresponding to
subordinators perturbed by an independent Brownian motion:
$$
\forall t \ge 0\,,\,D_t=G_t + \sigma B_t
$$
where $\{G_t,\  t\ge 0\}$ is a subordinator, i.e. a L\'evy process with non decreasing sample
paths. Since jumps of $\{D_t,\  t\ge 0\}$ are issued from $\{G_t,\  t\ge 0\}$ and are positive, we
recall that we say that $\{D_t,\  t\ge 0\}$ is {\it spectrally positive}.
This process can be characterized in terms of L\'evy exponents:
\begin{eqnarray*}
\forall u \in \RR, \; &\exp(t\phi_{D}(u))&= \EE[e^{iuD_{t}}]= \exp(t\phi_{G}(u)) \exp(t\phi_{B}(u))= \exp(t\phi_{G}(u)) \exp(-\frac{1}{2}t u^2\sigma^2)\\
&\phi_{G}(u)&=i\hat{\mu} u+\int_{\nbR\setminus \{0\}} [e^{iux}-1-iux \nbu_{[-1,1]}(x)]Q(dx)
\end{eqnarray*}
Exponent $\phi_{B}$ is associated to the Brownian motion and $\phi_{G}$ to $G_t$, which is in all
generality a jump process. It follows that the Lévy measure of $\{D_t,\  t\ge 0\}$ is the same as
that of $\{G_t,\  t\ge 0\}$ that we will denote by $\nu_{D}(dx)= Q(dx)$. Furthermore we will
suppose that measure $Q(.)$ admits a density with respect to the Lebesgue measure, i.e. that
$Q(dx)=q(x)dx$ for some density $q(.)$. In the following we will also need
$$
\varphi_D(u) = \phi_{D}(iu)= \varphi_G(u)+\frac{1}{2}u^{2}\sigma^{2},
$$
i.e. $\varphi_D(u)$ is such that $\EE[e^{-uD_{t}}]=e^{t\varphi_D(u)}$. We recall, since
$\{G_t,\  t\ge 0\}$ is a subordinator, that may write in this case $\varphi_D(u)$ in the following
way
$$\varphi_D(u)=-\mu u +\int_{0}^\infty [e^{-ux}-1]Q(dx)+\frac{1}{2}u^{2}\sigma^{2},
$$
for some $\mu\ge 0$. We consider in this paper several approaches for modelling degradation of a
device and its failure time. Failure time can traditionally be derived from a degradation model by
considering the first hitting time $T_b$ of a critical level $b>0$. The first-passage time
distribution has been already derived for the particular case of two sub-models. In the case of
Brownian motion with drift (corresponding to $G_t=\mu t$, $\mu>0$), it is the well-known inverse
Gaussian distribution, see \cite{FC} for instance. For the pure gamma process (i.e. $\sigma=0$ and
$\{G_t,\ t\ge 0\}$ is a gamma process), it has been studied by Park and Padgett \cite{ParkPadgett}.
Moreover they proposed an approximation for the cumulative distribution function of the hitting time based on Birnbaum-Saunders and inverse Gaussian  distributions.

Recently a new approach to define the failure time was proposed by Barker and Newby
\cite{Barker_Newby} that consists in considering the {\it last} passage time of degradation process
$\{D_t,\ t\ge 0\}$ above $b$. As explained in that paper, this is motivated by the fact that, even
if $\{D_t,\  t\ge 0\}$ reaches and goes beyond $b$, resulting in a temporarily "degraded" state of the device, it
can still always recover by getting back below $b$ provided this was not the last passage time
through $b$. On the other hand, if this is the last passage time then no recovery is possible
afterwards and we may then consider it as a "real" failure time. Of course, this discussion about
modelling failure time by the first or last passage time becomes irrelevant whenever process
$\{D_t,\  t\ge 0\}$ has non decreasing paths (which is not the case e.g. of the Brownian motion) since in that
case both quantities coincide.

In this paper we then investigate these quantities for a rather wide class of so-called perturbed
process. In Section \ref{Section_first_passage} we provide the Laplace transform of the first
passage time $T_b$ with penalty function involving the corresponding under and overshoot of the
process. We then confront this approach to related recent existing results on such passage times
distributions in the general theory of L\'evy processes, that introduces the notion of so-called
{\it scale functions}. The case of several sub-models is reviewed (or revisited) : in these cases
the probability distribution function (pdf) and/or the cumulative distribution function (cdf) can
be computed explicitly, or at least numerically. In conclusion of this section we propose an
alternative degradation process that takes into account the fact that the process cannot be in
theory negative and suggests that $\{D_t,\ t\ge 0\}$ be reflected at zero. In that setting we use
the aforementioned recent results in the theory of L\'evy and reflected L\'evy processes to obtain
the joint distribution of the first passage time jointly to the overshoot distribution. In Section
\ref{Section_last_passage} we study the case where failure time corresponds to the last passage
time $L_b$ above $b$ and derive its distribution in the non reflected and reflected case. Finally
we consider in Section \ref{Section_maintenance} a maintenance policy problem inspired by
\cite{Barker_Newby} and derive distribution of related quantities.

To conclude this introduction, we make precise where in the present paper previously published
results are reviewed and what is actually novel. Proposition \ref{prop_renewal_LT} in Section
\ref{Section_General_Case} is new, but its proof is similar to the one corresponding to proof of
Remark 4.1 as well as Expressions (4.4) and (4.5) of Garrido and Morales \cite{Garrido_Morales}.
Section \ref{Section_reflected} recalls facts (with short proofs) previously established in the
literature that are useful later on. On the other hand and to the best of our knowledge, Theorems
\ref{thm_Last_D_T} and \ref{thm_Last_D_T*_LT} in Section \ref{Section_last_passage} concerning last
passage times may be linked to Chiu and Yin \cite{Chiu_Yin}, Baurdoux \cite{Baurdoux} and recent
paper Kyprianou {\em et al.} \cite{KPR} but are otherwise genuinely new. Similarly Section
\ref{Section_maintenance} deals with determining reliability quantities features unheard-of
results.


\section{First-passage time as failure time}\label{Section_first_passage}

We consider here the hitting time distribution of a fixed level $b>0$ by the perturbed process
$\{D_t,\ t\ge 0\}$:
$$
T_b = \inf \left\{ t \ge 0 \;;\; D_t \ge b \right\}
$$
which we remind is a.s. finite. We study below the distribution of $(T_b,D_{T_b-},D_{T_b})$ by determining the following quantity
\begin{equation}\label{LT_penalty}
\phi_w(\delta,b)=\EE(e^{-\delta T_b}w(D_{T_b-}, D_{T_b}))
\end{equation}
where $\delta \ge 0$ and $w(.,.)$ is an arbitrary continuous bounded function that will be
referred to as {\it penalty function}. In the following we will drop the subscript when there is
no ambiguity on $w(.,.)$ and then write $\phi(\delta,b)$ instead of $\phi_w(\delta,b)$. We then
determine (\ref{LT_penalty}) in the general case and then illustrate our results to sub-models,
some of which distribution of $T_b$ has already been obtained.

\subsection{General case}\label{Section_General_Case}

We are interested in the case where process $\{G_t,\  t\ge 0\}$ is general. To this end, we use a well known technique that consists in approaching  the jump part process in $\{G_t,\  t\ge 0\}$ by a compound Poisson
processes which, as said in the Introduction, is similar to the one used in \cite{Garrido_Morales}
(for more details see Appendix A.1 in \cite{Garrido_Morales}). More precisely this process can be
pointwise approximated by a sequence of compound Poisson processes $((S(t,n))_{t \ge
0})_{n\in \NN}$ such that:
\begin{enumerate}
\item $(S(t,n))_{n\in \NN}$ is increasing for all $t\ge 0$,
\item $\mu t + \lim_{n\to \infty} S(t,n)=G_t$ for all $t\ge 0$,
\item for all $n$, $(S(t,n))_{t\ge 0}$ has intensity $\lambda_n$ and jump size with c.d.f. $P_n(x)$ with
\begin{eqnarray}
\lambda_n & = & \bar{Q}(1/n)\label{lambda_n}\\
P_n(x) & =& \frac{\bar{Q}(1/n)-\bar{Q}(x)}{\bar{Q}(1/n)}\nbu_{\{ x\ge 1/n\}}\label{P_n}
\end{eqnarray}
where $\bar{Q}(x):=Q([x,+\infty))$. Note that $\bar{Q}$ defines measure such that $\bar{Q}(dx)=-Q(dx)$.
\end{enumerate}
Note that this approach is particularly interesting when $\lambda_n=\bar{Q}(1/n)\longrightarrow
\bar{Q}(0)=Q([0,+\infty))=+\infty$ as $n\to\infty$, i.e. when process has infinitely many jumps on
any interval. Intuitively $\{S(t,n),\ t\ge 0\}$ is obtained from $\{G_t,\  t\ge 0\}$ by discarding
all jumps that are of size less than $1/n$. Since $\{S(t,n),\ t\ge 0\}$ increases towards $\{D_t,\
t\ge 0\}$, we have that
\begin{equation}
T_b^n \searrow T_b,\quad n\to \infty, \mbox{ a.s.},
\label{Conv}
\end{equation}
where $T_b^n$ is the hitting time of level $b$ of the truncated process $\{D^n_t,\  t\ge 0\}$ defined by $D^n_t=S(t,n)+\sigma B_t$ for any $t \ge 0$ and any $n \in \NN$. We remind that $T_b^n$ is also a.s. finite. In fact  $T_b^n$ may be described as a ruin time (i.e. the first hitting time of $0$ of a stochastic process) in the following way:
$$
T_b^n=\inf\{ t\ge 0\;;\; b-\mu t-S(t,n)+\sigma B_t<0 \}
$$
and we are interested in the Laplace transform $\phi_n(\delta):=\EE(e^{-\delta T^n_b}w(D^n_{T^n_b-}, D^n_{T^n_b}))$ of $T_b^n$ with penalty function $w(.,.)$ for all $\delta\ge 0$. Let $\rho_n=\rho_n(\delta)$ be the positive solution to the following equation:
\begin{equation}\label{Lundberg_n}
\lambda_n\int_0^\infty e^{-\rho_n x}dP_n(x)=\lambda_n+\delta -\frac{\sigma^2}{2}\rho_n^2+\mu \rho_n
\end{equation}
that we will call {\it generalized Lundberg equation}. 
We start by showing convergence of $\rho_n$ as $n\to \infty$.
\begin{prop}
$\rho_n$ converges as $n\to \infty$ to the unique solution $\rho>0$ to the following generalized Lundberg equation:
\begin{equation}\label{gen_Lundberg}
\delta - \frac{\sigma^2}{2}\rho^2=\varphi_G(\rho)\quad \iff \quad \delta= \varphi_D(\rho)
\end{equation}
\end{prop}
\begin{proof}
Thanks to Expressions (\ref{lambda_n}) and (\ref{P_n}) of $\lambda_n$ and c.d.f. $P_n$, we may
rewrite (\ref{Lundberg_n}) in the following way
\begin{eqnarray*}
&& \int_{1/n}^\infty e^{-\rho_n x} Q(dx)= \bar{Q}(1/n) + \delta - \frac{\sigma^2}{2}\rho_n^2+\mu \rho_n\iff  \int_{1/n}^\infty e^{-\rho_n x} Q(dx)= \int_{1/n}^\infty Q(dx) + \delta - \frac{\sigma^2}{2}\rho_n^2+\mu \rho_n\\
&&\iff \delta - \frac{\sigma^2}{2}\rho_n^2 +\mu \rho_n+ \int_{1/n}^\infty\left( 1- e^{-\rho_n x}\right
)Q(dx)=0.
\end{eqnarray*}
Thus $\rho_n$ is the only positive solution to equation $f_n(z)=0$ where
$
f_n(z):=\delta - \frac{\sigma^2}{2}z^2 +\mu z+ \int_{1/n}^\infty\left( 1- e^{-z x}\right)Q(dx).
$
Let us note that $(f_n)_{n \in \mathbb{N}}$ increasingly converges pointwise towards
$$
f(z)= \delta - \frac{\sigma^2}{2}z^2 +\mu z+ \int_{0}^\infty\left( 1- e^{-z x}\right)Q(dx)=\delta-\varphi_D(z),
$$
so that $\rho_n$ converges increasingly towards $\rho^*:=\sup_{n\in\nbN} \rho_n$. Besides one can verify that $f(z)=0$ admits an unique solution on $(0,+\infty)$, which is solution $\rho$ to Equation (\ref{gen_Lundberg}). Thus $\rho^*$ is less than or equal to solution $\rho$ and we prove that we in fact have equality $\rho^*=\rho$ which is achieved by showing that $f(\rho^*)=0$. Indeed, using inequality $0\le 1-e^{-zx}\le zx$ for all $z,x\ge 0$ and since $f_n(\rho_n)=0$, we have
\begin{eqnarray}
|f(\rho^*)| & = & |f(\rho^*)- f_n(\rho_n) |\nonumber\le |f(\rho^*)- f(\rho_n) | + |f(\rho_n)- f_n(\rho_n) |\nonumber\\
&\le & |f(\rho^*)- f(\rho_n) | + \int_0^{1/n} \left( 1- e^{-\rho_n x}\right)Q(dx)\le  |f(\rho^*)- f(\rho_n) | + \rho_n \int_0^{1/n} x Q( dx)\nonumber\\
&\le & |f(\rho^*)- f(\rho_n) | + \rho \int_0^{1/n} x Q(dx)\quad \mbox{since } \rho_n \le \rho^*\le \rho.\label{GLE_0}
\end{eqnarray}
We recall that the fact that $\{G_t,\  t\ge 0\}$ is a subordinator (a non decreasing L\'evy process) implies that $\int_0^{\infty} (1\wedge x) Q(dx)<+\infty$ (see e.g. (2) p.72 of \cite{Bertoin}), hence $\int_0^{1/n} x Q(dx)\longrightarrow 0$. Remembering that $f$ is a continuous function, this implies that (\ref{GLE_0}) tends to zero as $n\to +\infty$, hence $f(\rho^*)=0$.
\end{proof}

The Laplace transform $\phi_n(\delta)$ with penalty function $w(.,.)$ of $T_b^n$ is given through the following which is a particular case of Theorem 2 of \cite{TW} adapted to our context:
\begin{theorem}\label{th_renewal_n}
Let $w(.,.)$ be a bounded continuous function and define
$$
\omega_n(x)=\int_x^\infty w(x,y-x)dP_n(y).
$$
Then $b\mapsto \phi_n(\delta,b):= \EE(e^{-\delta T^n_b}w(D^n_{T^n_b-}, D^n_{T^n_b}))$ satisfies the
renewal equation
\begin{equation}
\phi_n(\delta,b)= \phi_n(\delta,\cdot)\star g_n(\delta,\cdot)(b)+ h_n(\delta,b)
\label{eq_renewal_n}
\end{equation}
where functions $g_n(\cdot,\cdot)$ and $h_n(\cdot,\cdot)$ are defined by
\begin{eqnarray}
g_n(\delta,y)&=& \frac{2\lambda_n}{\sigma^2} \int_0^y e^{-[-2\mu/\sigma^2+\rho_n(\delta)] (y-s)} \int_s^\infty e^{-\rho_n(\delta) (x-s)}dP_n(x) ds\label{g_n}\\
h_n(\delta,y)&=& e^{-[-2\mu/\sigma^2+\rho_n(\delta)]y}+ \frac{2\lambda_n}{\sigma^2} \int_0^y e^{-[-2\mu/\sigma^2+\rho_n(\delta)] (y-s)} \int_s^\infty e^{-\rho_n(\delta) (x-s)}\omega_n(x) dx ds. \label{h_n}
\end{eqnarray}
\end{theorem}
\begin{proof}
With notations of \cite{TW}, we have $g_n(\delta,y)$ expressed as in (1.10) therein with
$b:=b(\delta)=-2\mu/\sigma^2+\rho_n(\delta)$, $\lambda:=\lambda_n$, $P(\cdot):=P_n(\cdot)$ and $D= \sigma^2/2$.
Still with notations of \cite{TW}, and in Theorem 2 therein, we see that function $y\mapsto h_n(\delta,y)$ is the sum of
$e^{-[-2\mu/\sigma^2+\rho_n(\delta)]y}$ and some function $g_w(\cdot)$ defined in Expression (2.8) of \cite{TW} that depends on
$\omega_n$. It is easy to verify that this function is the last term on the right-handside of
(\ref{h_n}).
\end{proof}

Passing on the limit $n\to +\infty$ in Theorem \ref{th_renewal_n} yields the following renewal
equation for function (\ref{LT_penalty}):
\begin{prop}\label{prop_renewal_LT}
Let $\omega(x):=\int_x^\infty w(x,y-x)Q(dy)$. Function $\phi(\delta,\cdot)=\phi_w(\delta,\cdot)$
satisfies the renewal equation
\begin{equation}
\phi(\delta,b)=\phi(\delta,\cdot)\star g(\delta,\cdot)(b)+ h(\delta,b)
\label{eq_renewal_LT}
\end{equation}
where functions $g(\cdot,\cdot)$ and $h(\cdot,\cdot)$ are defined by
\begin{eqnarray}
g(\delta,y)&=& \frac{2}{\sigma^2} \int_0^y e^{-[-2\mu/\sigma^2+\rho(\delta)] (y-s)} \int_s^\infty e^{-\rho(\delta) (x-s)} Q(dx) ds\label{g}\\
h(\delta,y)&=& e^{-[-2\mu/\sigma^2+\rho(\delta)]y}+ \frac{2}{\sigma^2} \int_0^y e^{-[-2\mu/\sigma^2+\rho(\delta)] (y-s)} \int_s^\infty e^{-\rho(\delta) (x-s)}\omega(x) dx ds. \label{h}
\end{eqnarray}
Hence $\phi(\delta,b)$ is given by the Pollaczek-Kinchine like formula
\begin{equation}\label{PK_phi}
\phi(\delta,b)=\sum_{k=0}^\infty g^{\star k}(\delta,.)\star h(\delta,.) (\delta, b).
\end{equation}
\end{prop}
Note that (\ref{PK_phi}) is analogous to Expression (4.2) in \cite{Garrido_Morales}.\\

\begin{proof}
Let us prove that $\lambda_n \omega_n$ converges to $\omega$. This is easily seen by remembering that $\lambda_n=\bar{Q}(1/n)$ and thus that, by (\ref{P_n}),
$$
\lambda_n \omega_n(x)=-\int_x^\infty  w(x,y-x)\nbu_{\{ y\ge 1/n\}} d\bar{Q}(y)
$$
which converges to the desired expression, remembering that $-d\bar{Q}(y)=dQ(y)$. Convergence of $h_n$ to $h$ follows from (\ref{h_n}). In the same way, $\lambda_n\int_s^\infty e^{-\rho_n(\delta) (x-s)}dP_n(x)$ converges to $\int_s^\infty e^{-\rho(\delta) (x-s)} Q(dx)$, yielding convergence of $g_n$ to $g$ thanks to (\ref{g_n}).
\end{proof}

As announced in the Introduction, it is also possible to use the theory of L\'evy processes to propose a different approach for
determining the joint distribution of the hitting time $T_b$ jointly to the state of $D_{T_b}$, using {\it scale functions}.
More precisely, we have the following proposition from e.g. Kyprianou and Palmowski \cite{KP}:
\begin{prop}[Theorem 1 (4) \cite{KP}]\label{approach_KP}
Let us define for all $\delta \ge 0$ the scale function $W^{(\delta)}$, through its Laplace transform, and
$Z^{(\delta)}$ by
\begin{eqnarray}
 \int_0^\infty e^{-\lambda x} W^{(\delta)}(x)dx&=&  \frac{1}{\varphi_D(\lambda)-\delta},\quad \lambda >
 \rho(\delta) \label{eq:Wq}\\
 Z^{(\delta)}(x)&=& 1+\delta \int_0^x W^{(\delta)}(y)dy\label{eq:Zq},
\end{eqnarray}
where we recall that $\rho(\delta)$ is solution to the Lundberg equation
$\varphi_D(\lambda)=\delta$. Then from Expression (4) p.19 of \cite{KP} one has that
\begin{equation}\label{KP_T_b}
\EE[e^{-\delta T_b}]= Z^{(\delta)}(b) - \frac{\delta}{\rho(\delta)}W^{(\delta)}(b).
\end{equation}
\end{prop}
Just to be clear on notations, we emphasize that \cite{KP} deals with spectrally negative processes. To apply it here (hence to obtain Expressions (\ref{eq:Wq}), (\ref{eq:Zq}) and (\ref{KP_T_b})), we thus need to consider hitting time of $0$ of process $\tilde{D}_t:=-D_t$ starting from $\tilde{D}_0=b$. In particular, Laplace exponent $\psi(.)$ of $\tilde{D}_t$ as defined in Expression (2) of \cite{KP} by $\EE[e^{\lambda \tilde{D}_t}]=e^{t\psi(\lambda)}$ does coincide with function $\varphi_D(.)$, and $\Phi(\delta)=\sup\{\lambda\ge 0|\ \psi(\lambda)=\delta \}$, also defined in \cite{KP}, coincides with $\rho(\delta)$.
\begin{remark}[scale function regularity]\label{scale_diff}
A necessary condition for function $W^{(\delta)}$ defined in the Proposition \ref{approach_KP} to
be differentiable is that $\{D_t,\  t\ge 0\}$ has unbounded variation, which is the case here since
it has a Gaussian component (i.e. $\sigma>0$). In fact it is shown in \cite{CKS} the stronger fact
that $\sigma>0$ implies that $W^{(\delta)}$ is twice differentiable.
\end{remark}
\begin{remark}[boundary value of scale function]\label{scale_boundary}
Still in the present case where process $\{D_t,\  t\ge 0\}$ has unbounded variation, we have that $W^{(\delta)}(0)=0$ by Lemma
8.6 p.222 of \cite{Kyprianou_book}.
\end{remark}
As a complement to (\ref{KP_T_b}), it is interesting to note that Remark 3 of \cite{KP} gives an explicit expression of the joint Laplace transform of $(T_b, D_{T_b})$.

The approach in Proposition \ref{approach_KP} has however a cost, which is that a Laplace Transform
inversion of (\ref{eq:Wq}) is required to obtain the scale function. However recent results have
been found concerning expression of $W^{(\delta)}$ in particular cases, see Hubalek and Kyprianou
\cite{Hubalek_Kyprianou} as well as Egami and Yamazaki \cite{Egami} in the case where $\{G_t,\ t\ge
0\}$ is a compound Poisson process with jumps following phase-type distribution. In fact the
following result combines both approaches given in Propositions \ref{prop_renewal_LT} and
\ref{approach_KP}, and theoretically gives a closed form expression of scale function
$W^{(\delta)}$ of any spectrally positive L\'evy process:
\begin{prop}\label{expr_scale_fc}
Scale function $W^{(\delta)}$ uniquely defined by Laplace transform (\ref{eq:Wq}) satisfies the
following first order differential equation
\begin{equation}\label{EDO_W}
W^{(\delta)'}(x)-\rho(\delta) W^{(\delta)}(x)=-\frac{\rho(\delta)}{\delta}\sum_{k=0}^\infty
g^{\star k}(\delta,.)\star h'(\delta,.) (\delta, x):=H(\delta,x)
\end{equation}
where $g(\delta,.)$ is given by (\ref{g}) and $h'(\delta,.)$ is derivative of $h(\delta,.)$ given
in (\ref{h}) with $w\equiv 1$, i.e.
\begin{multline}\label{hprime}
h'(\delta,y)=-[-2\mu/\sigma^2+\rho(\delta)] e^{-[-2\mu/\sigma^2+\rho(\delta)]y}+\frac{2}{\sigma^2}\int_y^\infty e^{-[-2\mu/\sigma^2+\rho(\delta)]
(x-y)}\bar{Q}(x) dx\\
-[-2\mu/\sigma^2+\rho(\delta)]\frac{2}{\sigma^2} \int_0^y e^{-[-2\mu/\sigma^2+\rho(\delta)] (y-s)} \int_s^\infty
e^{-\rho(\delta) (x-s)}\bar{Q}(x) dx ds.
\end{multline}
Thus $W^{(\delta)}(x)$ has the following explicit expression
\begin{equation}\label{exp_W}
W^{(\delta)}(x)=\int_0^x e^{-\rho(\delta)(x-y)}H(\delta,y)dy.
\end{equation}
\end{prop}
\begin{proof}
Differential equation (\ref{EDO_W}) simply comes from (\ref{KP_T_b}) that one differentiates with
respect to $b$ (which is possible since $W^{(\delta)}$ is differentiable in light of Remark
\ref{scale_diff}), using expression $\EE[e^{-\delta T_b}]=\phi_w(\delta,b)$ where penalty function
$w(.)$ is identically equal to $1$, and finally using Expression (\ref{PK_phi}). Note that
differentiation of (\ref{PK_phi}) is done by using the well known property of derivation of
convoluted functions $(f\star g)'=f'\star g=f\star g'$, explaining why $H(\delta,.)$ features
derivative of function $h(\delta,.)$.
\\
Since by Remark \ref{scale_boundary} one has that $W^{(\delta)}(0)=0$, Equation (\ref{exp_W}) is then obtained by
solving the standard first order differential equation (\ref{EDO_W}).
\end{proof}

Note however that Formula (\ref{exp_W}) requires to compute the infinite series appearing in (\ref{EDO_W}), which in practice may not be handy. However, since such scale functions are important in the theory of L\'evy processes (in particular, these functions will be useful in Sections \ref{Section_reflected} and  \ref{Section_last_passage} for determining quantities related to first passage times of reflected processes and last passage times), any expression can be considered as welcome.

\begin{remark}\label{TCL}
Asymptotic behaviour of $T_b$ as $b\to +\infty$ may be obtained through Roynette et al \cite{RVV}.
More precisely, it was proved that $(T_b+b/\varphi_D'(0))/\sqrt{b}$ converges in distribution to an
${\cal N}(0,- \varphi_D''(0)/\varphi_D'(0)^3)$ distribution. One can also compute from \cite{RVV}
asymptotic behaviour of triplet $\left(T_b+b/\varphi_D'(0))/\sqrt{b},\ D_{T_b}-b,\
b-D_{T_b-}\right)$that we did not include here but that involve technical expressions.
\end{remark}

\subsection{Examples}\label{subsection_examples}

We illustrate the previous study with examples and review some famous examples related to
degradation models.

\paragraph{Pure gamma process} Here we assume that $\sigma=0$ and that $\{G_t,\  t\ge 0\}$ is a gamma process with shape parameter $\alpha$ and scale parameter $\xi$. We recall that its L\'evy exponent and L\'evy measure are given
by
\begin{eqnarray*}
\varphi_G(u) & = &  \varphi_D(u)=-\alpha\log(1+u/\xi)\\
\nu_D(dx) & = & Q(dx)  =  x^{-1} e^{-\frac{x}{\xi}}\alpha dx.
\end{eqnarray*}
Considering this special case into the generalized Lundberg equation, it follows that this equation
has no positive solution. It appears that the presence of the perturbation in  the degradation
model is important for applying the result obtained by Tsai and Wilmott \cite{TW} as we did in Proposition
\ref{prop_renewal_LT}. However, in this first special case, the degradation process reduces to a
pure stationary gamma process and so $\{D_t,\  t\ge 0\}$ has increasing paths. It follows that:
$$
\forall t \ge 0 \;, \quad \PP[T_b>t] = \PP[D_t<b].
$$
Consequently it is sufficient to study the distribution of $D_t$ for any $t \ge 0$. The hitting
time distribution was already given for instance p.517 of Park and Padgett \cite{ParkPadgett}:

\begin{prop}[Park and Padgett \cite{ParkPadgett}]
The cumulative distribution function (cdf) of $T_b$ is:
$$
\forall t \ge 0 \;, \quad F(t) = \frac{\Gamma(\alpha t,b/\xi)}{\Gamma(\alpha t)},
$$
where $\Gamma(\cdot,\cdot)$ is the upper incomplete Gamma function. The probability distribution
function (pdf) of $T_b$ is, for any $t \ge 0$:
$$
f(t) = \alpha \left( \Psi(\alpha t)-\log\left(\frac{b}{\xi}\right)\right)\frac{\gamma(\alpha t,b/\xi)}{\Gamma(\alpha t)}
+ \frac{\alpha}{(\alpha t)^2 \Gamma(\alpha t)}\left( \frac{b}{\xi} \right)^{\alpha t} {}_2F_2(\alpha t,\alpha t;\alpha t+1,\alpha t+1;-b/\xi),
$$
where $\Psi$ is the di-gamma function (or logarithmic derivative of the Gamma function), $\gamma(\cdot,\cdot)=\Gamma(\cdot)-\Gamma(\cdot,\cdot)$ is the lower incomplete Gamma function and ${}_2F_2$ the generalized hypergeometric function of order $(2,2)$.
\end{prop}

It has been proved (see \cite{AH:survey} or Section 5 of \cite{ShakedShanki} for instance) that
$T_b$ has an increasing failure rate. 

\paragraph{Perturbed gamma Process} Statistical inference in a perturbed gamma process has been studied in \cite{BPS} using only
degradation data. However sometimes both degradation data and failure time data are available (see
\cite{Axel} for such problem for a related model). In addition, from parameters estimation (based
on degradation data for instance), one can obtain an estimation of the failure time distribution.
Hence the distribution of $T_b$ can be of interest. In that case, $\{G_t,\  t\ge 0\}$ is a gamma process with
shape parameter $\alpha$ and scale parameter $\xi$. We recall that L\'evy exponent and L\'evy
measure of process $\{D_t,\  t\ge 0\}$ are then given by
\begin{equation}\label{car_pert_gamma}
\begin{array}{rcl}
\varphi_D(u)&=& -\alpha\log(1+u/\xi)+\frac{1}{2}u^{2}\sigma^{2}\\
\nu_D(dx) & = & Q(dx)  =  x^{-1} e^{-\frac{x}{\xi}}\alpha dx.
\end{array}
\end{equation}
Thus, Proposition \ref{prop_renewal_LT} gives joint distribution of $(T_b, D_{T_b-}, D_{T_b})$
through expression of $\phi(\delta,b)$ where $\omega(x):=\int_x^\infty w(x,y-x)\frac{e^{-y/\xi}}{y}
dy$ and $g(\delta,y)= \frac{2}{\sigma^2} \int_0^y e^{-\rho(\delta) (y-s)} \int_s^\infty
e^{-\rho(\delta) (x-s)} \frac{e^{-x/\xi}}{x}dx ds$, $w(.,.)$ being an arbitrary bounded function.
Also note that from Remark \ref{TCL} one has thanks to \cite{RVV} the Central Limit Theorem
$$
\frac{T_b-\xi b/\alpha}{\sqrt{b}}\stackrel{\cal D}{\longrightarrow} {\cal N}\left(0, \frac{\alpha/\xi^2+\sigma^2}{\alpha^3/\xi^3}\right), \quad b\to +\infty.
$$
Finally, expression of the scale function is then given by (\ref{exp_W}) with $\varphi_D(.)$ and $Q(.)$ defined in (\ref{car_pert_gamma}). This will come in handy in Section \ref{Section_last_passage}.

\paragraph{Brownian motion with positive drift} We consider the case where $G_t=\mu t$, i.e. $\{D_t,\  t\ge 0\}$ is a Brownian motion with drift. In such case, the distribution of the hitting time of the constant boundary $b$ is known and is called the inverse Gaussian distribution. Its pdf is given by:
$$
\forall t \ge 0 \;, f(t) = \frac{b}{\sqrt{\sigma^2 t^3}} \exp\left( -\frac{(b-\mu t)^2}{2t \sigma^2} \right) \;.
$$
Proof of this result is generally based on the symmetric principle full-filled by the Brownian motion when $\mu=0$, or can be showed with martingale methods in the case $\mu>0$. Alternatively the pdf can be obtained  by inverting the Laplace transform of $T_b$:
\begin{equation}\label{ex_BM}
\phi(\delta) = \EE[e^{-\delta T_b}] = \exp\left( -\frac{(\gamma_\delta-\mu)b}{\sigma^2} \right) \;,
\end{equation}
with $\gamma_\delta = \sqrt{\mu^2 + 2\delta\sigma^2}$. Note that the expression of this Laplace
transform is standard and can be found e.g. in Expression (38) p. 212 of \cite{Cox_Miller} (see also \cite{AH:survey}, page 19). Also
note that (\ref{ex_BM}) is compatible with Expression (\ref{PK_phi}). Indeed in the context of
Proposition \ref{prop_renewal_LT} we have here $g\equiv 0$ and $h\equiv 0$, thus (\ref{PK_phi})
reduces to $\phi(\delta,b)=e^{-[-2\mu/\sigma^2+\rho(\delta)]y}$ where $\rho(\delta)$ satisfies
(\ref{gen_Lundberg})$\iff 0=\frac{\sigma^2}{2}\rho(\delta)^2-\mu \rho(\delta) -\delta$, giving the
exact same expression (\ref{ex_BM}).

Expression of the scale function for this case is then given e.g. p.121 in \cite{Hubalek_Kyprianou} by
\begin{equation}\label{scale_BM}
W^{(\delta)}(x)=\frac{2}{\sqrt{2\delta \sigma^2+\mu^2}} e^{-\mu x/\sigma^2}\sinh \left(
\frac{x}{\sigma^2} \sqrt{2\delta \sigma^2+\mu^2}\right)=\frac{2}{\gamma_\delta} e^{-\mu
x/\sigma^2}\sinh \left( \frac{x}{\sigma^2} \gamma_\delta\right).
\end{equation}
Note that there seems to be a small mistake in \cite{Hubalek_Kyprianou} of expression of $W^{(\delta)}(x)$ (where there are some $\mu$'s instead of $\mu^2$'s), that we corrected here. As proved by Chhikara and Folks \cite{CF77}, the failure rate of an inverse Gaussian distribution
is non-monotone, but it is initially increasing and then decreasing.

\paragraph{Perturbed compound Poisson process with phase-type jumps} Let us suppose that $\{G_t,\ t\ge 0\}$
is a compound Poisson process of intensity $\lambda$ whose jumps are phase-type distributed with representation
$(m,{\bf \alpha},{\bf T})$. Let ${\bf t}:=- {\bf T} {\bf 1}$ where ${\bf 1}$ is a column vector
of which entries are equal to $1$'s of appropriate dimension (see e.g. Chapter VIII of Asmussen
\cite{Asmussen_ruinproba} for an extensive account on such distributions). In that case $\varphi_D$
is given by
$$
\varphi_D(u)=\frac{1}{2}u^{2}\sigma^{2}+\lambda({\bf \alpha}(uI-{\bf T})^{-1} {\bf t}-1).
$$
Egami and Yamazaki \cite{Egami} give the expression of the Laplace transform $\EE(e^{-\delta T_b})$
by determining a closed formula for the scale functions $W^{\delta}$ and using results in
Proposition \ref{approach_KP}. More precisely following \cite{Egami}, let us denote for all
$\delta>0$ the complex solutions $(\xi_{i,\delta})_{i}$ (resp. $(\eta_i)_i$) of Equation
$\varphi_D(u)=\delta$ (resp. $\delta/(\delta-\varphi_D(u))=0$), $u\in \nbC$. We suppose that the
$\xi_{i,\delta}$'s are distinct roots. We set
\begin{eqnarray*}
{\cal I}_\delta &:= & \{ i|\ \varphi_D(-\xi_{i,\delta})=\delta\mbox{ and } \Re (\xi_{i,\delta})>0\},\\
{\cal J}_\delta &:= & \{ i|\ \delta/(\delta-\varphi_D(-\eta_{i}))=0\mbox{ and } \Re (\eta_{i})>0\},\\
\varphi_\delta^- (u) & =& \frac{\prod_{j\in {\cal J}_\delta}(u+\eta_j)}{\prod_{j\in {\cal
J}_\delta}\eta_j} \frac{\prod_{i\in {\cal I}_\delta}\xi_{i,\delta}}{\prod_{i\in {\cal
I}_\delta}(u+\xi_{i,\delta})}.
\end{eqnarray*}
On page 4 of \cite{Egami} it is stated that $\mbox{Card}({\cal I}_\delta)=\mbox{Card}({\cal J}_\delta)+1$ (this
results in fact comes from Lemma 1 (1) of \cite{Asmussen_Avram_Pistorius}), so that
$\varphi_\delta^- (\infty)$ exists and is equal to $0$. We then define
\begin{eqnarray*}
(A_{i,\delta})_{i\in {\cal I}_\delta}& \mbox{ s.t. }& 
\varphi_\delta^- (u)-\varphi_\delta^- (\infty)=\varphi_\delta^- (u)=\sum_{i\in {\cal I}_\delta} A_{i,\delta} \frac{\xi_{i,\delta}}{\xi_{i,\delta}+u},\\
\varrho_\delta &:=& \sum_{i\in {\cal I}_\delta} A_{i,\delta}\xi_{i,\delta}.
\end{eqnarray*}
Then Proposition 2.1 of \cite{Egami} gives expression of the Laplace transform $
\phi(\delta)=\EE(e^{-\delta T_b})= \sum_{i\in {\cal I}_\delta} A_{i,\delta}e^{-\xi_{i,\delta}(x-a)}
$and Proposition 3.1 of \cite{Egami} yields the following interesting and useful expression of
the scale function
\begin{equation}\label{scale_PH}
W^{(\delta)}(x)=\frac{2}{\sigma^2\varrho_\delta}\sum_{i\in {\cal I}_\delta}
A_{i,\delta}\frac{\xi_{i,\delta}}{\rho(\delta)+\xi_{i,\delta}}\left[e^{\rho(\delta)x} -
e^{-\xi_{i,\delta} x}\right].
\end{equation}
Furthermore, as pointed out in \cite{Egami}, expressions of $W^{(\delta)}$ are more complicated but
available when roots $\xi_{i,\delta}$'s have multiplicity $m_i>1$.

\subsection{Reflected processes}\label{Section_reflected}

The previous model may not be too realistic if we consider the Brownian motion as a means of
modelling small repairs, as the degradation process $\{D_t,\  t\ge 0\}$ may then be negative. An alternative
for this is to consider the reflected version of $\{D_t,\  t\ge 0\}$ defined in the following way
$$
\forall t \ge 0,\quad D^*_t:=D_t-\inf_{0\le s \le t} (D_s\wedge 0).
$$
The hitting time distribution $T^*_b$ of $\{D^*_t,\  t\ge 0\}$ jointly to the overshoot and undershoot pdf is given by the following theorem
\begin{thm}\label{thm_reflected}
Let us suppose that $\{D_t,\ t\ge 0 \}$ is non monotone, i.e. that $\sigma>0$. Let $W^{(\delta)}$
be defined by (\ref{eq:Wq}) where we recall that $\rho=\rho(\delta)$ is solution to the Lundberg
equation $\varphi_D(z)=\delta$. Then
\begin{equation}\label{reflected_LT}
\EE[e^{-\delta T_b^*}\nbu_{\{D^*_{T_b^*-}\in dy,\ D^*_{T_b^*}\in dz\}}]=\nu_D(dz-y)
\hat{r}_b^{(\delta)}(b,y)dy
\end{equation}
where $\displaystyle \hat{r}_b^{(\delta)}(b,y):=\frac{W^{(\delta)}(b) W^{(\delta)'}(y)}{W^{(\delta)
'}(b)}-W^{(\delta)}(y)$.
\end{thm}

\begin{proof}
We apply results from Doney \cite{Doney} and we write, following notations therein, $X_t:=-D_t$, so
that L\'evy measure of $\{X_t,\  t\ge 0\}$ is $\Pi(dx):=\nu_D(-dx)$and process $\hat{Y}(t):=\sup_{0\le s \le
t} (X_s\vee 0)-X_t$ is equal to $D^*_t$. Following terminology of \cite{Doney}, $W^{(\delta)}$ is
the $\delta$-scale function of $\{X_t,\  t\ge 0\}$ and is defined by (\ref{eq:Wq}) with $\varphi_{-D}$ instead of $\varphi_{D}$. Remark 4 p.14 of
\cite{Doney} gives expression (\ref{reflected_LT}) where $\hat{r}_b^{(\delta)}$ is given by
Pistorius \cite{Pistorius} (see also Expression (15) in Theorem 1 of \cite{Doney}) with $x:=0$ and
$a:=b$, noting that function $W^{(\delta)}$ is differentiable by Remark \ref{scale_diff}.
\end{proof}

Note again that Theorem \ref{thm_reflected} is especially interesting when function $W^{(\delta)}$
admits closed form expressions, as in \cite{Hubalek_Kyprianou, Egami}. For example in the case of a
perturbed compound Poisson process with phase-type distributed jumps (and using the same notations as in Section
\ref{subsection_examples}) we have $\nu_D(dz)=\lambda {\bf \alpha} e^{{\bf T}z}{\bf t}$ and
$W^{(\delta)}$ given by  (\ref{scale_PH}) (of which derivative is easily available), which, plugged
in (\ref{reflected_LT}), easily yields the Laplace transform of the corresponding hitting time
$T_b^*$ jointly to the overshoot and undershoot distribution.

We now state a famous lemma that links distribution of $D^*_t$ to the cumulative distribution function of $T_b$ for all $b \ge 0$:
\begin{lemm}\label{duality_lemma}
We have for all $b$ and $t \ge 0$, $\PP(D^*_t>b)=\PP(T_b\le t)$.
\end{lemm}

\begin{proof}
This is a simple consequence from e.g. Lemma 3.5 p.74 of Kyprianou \cite{Kyprianou_book}that implies that $\PP(D^*_t>b)=\PP\left(\sup_{0\le s \le t}D_s>b\right)$ which in turn is equal to $\PP(T_b\le t)$.
\end{proof}

\section{Last-passage time as failure time}\label{Section_last_passage}

We let $L_b$ and $L^*_b$ be the last passage times of processes $\{D_t,\  t\ge 0\}$ and $\{D^*_t,\  t\ge 0\}$ below level $b$ defined as
$$
L_b:=\sup\{ 0\le u|\ D_u \le b \} \quad  {\mbox{and}} \quad L_b^*:= \sup\{ 0\le u|\ D_u^* \le b \}
$$
which are well defined as processes $\{D_t,\  t\ge 0\}$ and $\{D^*_t,\  t\ge 0\}$ satisfy $\lim_{t\to +\infty} D_t=
\lim_{t\to +\infty} D_t^*=+\infty$.

\subsection{General case}

Let us introduce the following bivariate measures $\cal U$ and $\hat{\cal U}$ on $[0,+\infty)^2$
through their double Laplace transforms
\begin{equation}\label{calU}
\int_0^\infty \!\!\!\int_0^\infty e^{-\alpha s -\beta x} {\cal
U}(ds,dx)=\frac{\rho(\alpha)-\beta}{\alpha-\varphi_D(\beta)},\ \forall\beta>\rho(\alpha),\quad
\int_0^\infty \!\!\!\int_0^\infty e^{-\alpha s -\beta x} \hat{\cal
U}(ds,dx)=\frac{1}{\rho(\alpha)+\beta},\ \forall \beta,\alpha \ge 0.
\end{equation}
Expressions (\ref{calU}) may be found in Expressions (12) and (13) of \cite{Biffis_Morales}, or
p.154 and p.170 in Chapter 6 of \cite{Kyprianou_book} (note that the latter reference considers
spectrally negative processes, hence roles for ${\cal U}$ and $\hat{\cal U}$ are swapped therein).
Furthermore, from (26) of \cite{Biffis_Morales} one has that $\hat{\cal U}_\delta(dx):=
\int_{s=0}^\infty e^{-\delta s}\hat{\cal U}(ds,dx)$ has the expression
\begin{equation}\label{calU_hat}
\hat{\cal U}_\delta(dx)=e^{-\rho(\delta)x}dx,
\end{equation}
hence $\hat{\cal U}_\delta([0,+\infty))=1/\rho(\delta)$. In the same spirit, we define ${\cal
U}_\delta(dx):=\int_{s=0}^\infty e^{-\delta s}{\cal U}(ds,dx)$. (\ref{calU}) then reads that
$\int_{x=0}^\infty e^{-\beta x}{\cal
U}_\delta(dx)=\frac{\rho(\delta)-\beta}{\delta-\varphi_D(\beta)}$ for all $\beta>\rho(\delta)$. We
then have the following identity, that will be of interest later on.
\begin{lemm}\label{prop_calU_LT}
One has \begin{equation}\label{calU_LT}
{\cal U}_\delta(dx)=[- \rho(\delta)W^{(\delta)}(x) + W^{(\delta)'}(x)]dx.
\end{equation}
\end{lemm}
\begin{proof}
From (\ref{eq:Wq}) we get the following
\begin{equation}\label{to_integrate2}
\int_{x=0}^\infty e^{-\beta x}{\cal U}_\delta(dx)= - \rho(\delta) \int_0^\infty e^{-\beta x} W^{(\delta)}(x)dx + \int_0^\infty \beta e^{-\beta x} W^{(\delta)}(x)dx
\end{equation}
where $\beta>\rho(\delta)$. We recall from Remark \ref{scale_boundary} that $W^{(\delta)}(0)=0$. As
to behaviour at $+\infty$ of the scale function, we have, thanks to Lemma 8.4 p.222 of
\cite{Kyprianou_book}, relation $W^{(\delta)}(x)=e^{cx}W_c^{(\delta-\varphi_D(c))}(x)$, for any
$c\in\RR$ such that $\delta-\varphi_D(c)\ge 0$, where $W_c^{(\delta-\varphi_D(c))}$ is a scale
function defined under a different probability measure. By picking $c=\rho(\delta)$ then one gets
$\delta-\varphi_D(c)= 0$ and
\begin{equation}\label{shifted}
W^{(\delta)}(x)=e^{\rho(\delta)x}W_{\rho(\delta)}^{(0)}(x)
\end{equation}
(see e.g. Second Remark p.32 of \cite{Pistorius2} for this identity as well as details on this
other probability measure). At the end of Proof of Corollary 8.9 p.227 of \cite{Kyprianou_book}, it
is shown that $W_{\rho(\delta)}^{(0)}(+\infty)= \frac{1}{\varphi_{D,\rho(\delta)}'(0+)}$ where
$\varphi_{D,\rho(\delta)}(q):=\varphi_D(q+\rho(\delta))
-\varphi_D(\rho(\delta))=\varphi_D(q+\rho(\delta)) -\delta$, hence
\begin{equation}\label{shifted2}
W_{\rho(\delta)}^{(0)}(+\infty)= \frac{1}{\varphi_D '(\rho(\delta)) }<+\infty .
\end{equation}
Thus in view of $W^{(\delta)}(0)=0$, (\ref{shifted}) and (\ref{shifted2}), and since
$\beta>\rho(\delta)$, the following integration by parts makes sense:
\begin{equation}\label{IPP}
\int_0^\infty \beta e^{-\beta x} W^{(\delta)}(x)dx= \left[ -e^{-\beta x} W^{(\delta)}(x)\right]_0^\infty + \int_0^\infty e^{-\beta x} W^{(\delta)'}(x)dx = 0 + \int_0^\infty e^{-\beta x} W^{(\delta)'}(x)dx,
\end{equation}
remembering that $W^{(\delta)}$ is indeed differentiable by Remark \ref{scale_diff}. Comparing
Laplace transforms (\ref{to_integrate2}) and (\ref{IPP}), we then obtain (\ref{calU_LT}).
\end{proof}

Let us also note that, according to Definition 6.4 p.142 of \cite{Kyprianou_book}, the fact that $\sigma>0$ entails that $0$ is regular for sets $(-\infty,0)$ and $(0,+\infty)$ (in particular, Theorem 6.5 p.142 of \cite{Kyprianou_book} applies here). With that in mind, and since $\{D_t,\ t\ge 0\}$ is spectrally positive and drifts to $+\infty$, we may recall the following important recent result from Kyprianou {\it et al} \cite{KPR}.
\begin{thm}[Corollary 2 of Kyprianou, Pardo and Rivero \cite{KPR}]\label{Coro_KPR}
Let us define
\begin{eqnarray*}
\underline{D}_\infty:=\inf_{u\ge 0}D_u, &\quad & \underline{D}_s=\inf_{t\le s}D_s,\quad
\underline{G}_\infty:=\sup\{ s\ge 0|\ D_s-\underline{D}_s=0\},\\
\underrightarrow{D}_t:= \inf_{ s>t} D_s,&\quad &
\underrightarrow{\cal D}_t:= \inf\{ s>t|\ D_s-\underrightarrow{D}_t=0\}.
\end{eqnarray*}
Then distribution of $(\underline{G}_\infty, \underline{D}_\infty, \underrightarrow{\cal
D}_{L_b}-L_b, L_b, \underrightarrow{D}_{L_b}-b, b-D_{L_b-}, D_{L_b}-b)$ is given by the following
identity for $t,b,v>0$, $ s>r>0$, $0\le y<b+v$, $w\ge u>0$:
\begin{multline}
\PP[\underline{G}_\infty \in dr,\ -\underline{D}_\infty\in dv,\ \underrightarrow{\cal
D}_{L_b}-L_b\in dt,\ L_b\in ds,\ \underrightarrow{D}_{L_b}-b\in du,
 b-D_{L_b-}\in dy,\ D_{L_b}-b\in dw]\\
 =\hat{\cal U}_\delta([0,+\infty))^{-1}\hat{\cal U}(dr,dv){\cal U}(ds-r,b+v-dy)\hat{\cal U}(dt,w-du)Q(dw+y),\label{n_uple}
\end{multline}
where $\hat{\cal U}_\delta([0,+\infty))^{-1}=\rho(0)$ from (\ref{calU_hat}).
\end{thm}
It is clear that distribution of $(L_b, b-D_{L_b-}, D_{L_b}-b)$ may be theoretically obtained from
this theorem. In fact, our goal is to propose expressions of this distribution that only involves
quantities that were determined in Section \ref{Section_General_Case}, e.g. scale functions, which
we saw can be available in many situations, as opposed to measures $\cal U$ and $\hat{\cal U}$
appearing in (\ref{n_uple}) which, as seen in (\ref{calU}), are available only through double
Laplace transforms. More precisely, we have the following results.
\begin{thm}\label{thm_Last_D_T}
We have for all $t\ge 0$ and $a\in \RR$,
\begin{eqnarray}
\PP(L_b<t)                 & = & \int_b^\infty \EE[D_1]W(a-b)f_{D_t}(a)da \label{Last_D_T} \\
\PP(L_b\ge t, \ D_t\in da) & = & [1-\EE[D_1]W(a-b)]f_{D_t}(a)da \label{Last_D_T_density}
\end{eqnarray}
where $f_{D_t}(.)$ is density of r.v. $D_t$ and $W(.)=W^{(\delta)}(.)$ defined in (\ref{eq:Wq})
with $\delta=0$. Besides, for all $\delta\ge 0$, and for $b>y\ge 0$, $w>0$, the Laplace transform
of $L_b$ jointly to density of the under and overshoot is given by
\begin{equation}\label{Last_D_T_over_under}
\EE[e^{-\delta L_b}\nbu_{\{b-D_{L_b-}\in dy,\ D_{L_b}-b\in dw\}}]=\left[
e^{\rho(\delta)(b-y)}\frac{1}{\varphi_D '(\rho(\delta)) }-W^{(\delta)}(b-y)\right]dy
.[1-e^{-\rho(0)w}]Q(dw+y).
\end{equation}
\end{thm}
Let us compare results given in Theorem \ref{thm_Last_D_T} with existing ones in the literature
concerning last passage times of L\'evy processes. References \cite{Chiu_Yin} and \cite{Baurdoux}
give distributions of respectively last exit times and last exit times before an exponentially
distributed time, in terms of their Laplace transform, for a similar class of L\'evy processes;
however Theorem \ref{thm_Last_D_T} is more adapted here as it directly gives its cdf jointly to the
density of the overshoot, thus avoiding an inverse Laplace transform. As said before, the slight
advantage of Formula (\ref{Last_D_T_over_under}) over (\ref{n_uple}) is that it only involves the
scale function.

\begin{proof}
Let us start by showing (\ref{Last_D_T}) and (\ref{Last_D_T_density}). Let $t>0$. By definition of
$L_b$ we note that for all $a\ge b$ event $\left[  L_b <t, D_t\in da\right]$ is equal to
$\left[D_t\in da,\ \{D_s\} \mbox{ will not hit level }b\mbox{ anymore after } t\right]$. Hence
using the Markov property:
\begin{eqnarray*}
\PP\left[  L_b <t, D_t\in da\right]&=&\PP_{a-b} \left[  T_0=+\infty \right]\PP\left[ D_t \in da\right]
\end{eqnarray*}
where $\PP_{a-b} \left[  T_0=+\infty \right]$ is the probability that process $\{D_t,\  t\ge 0\}$
starting from $a-b$ will never hit $0$ and is given e.g. through Formula (4) p.19 of \cite{KP} by
$\PP_{a-b} \left[  T_0=+\infty \right]=\EE[D_1]W(a-b)$ and $\PP\left[ D_t \in
da\right]=f_{D_t}(a)da$ where $f_{D_t}$ is the density of r.v. $D_t$ and $W(.)=W^{(0)}(.)$ in
(\ref{eq:Wq}). By integrating $a$ from $b$ to $+\infty$ one gets (\ref{Last_D_T}). Equation
(\ref{Last_D_T_density}) stems from the basic equality $\PP(L_b\ge t, \ D_t\in da)=\PP(D_t\in da)-
\PP(L_b< t, \ D_t\in da)$.
\\
We now turn to (\ref{Last_D_T_over_under}), and use Theorem \ref{Coro_KPR} to this end. Since by
Fubini theorem we have
$$
\EE[e^{-\delta L_b}\nbu_{\{b-D_{L_b-}\in dy,\ D_{L_b}-b\in dw\}}]= \int_{s=0}^\infty e^{-\delta
s}\PP[L_b\in ds, b-D_{L_b-}\in dy,\ D_{L_b}-b\in dw],
$$
and in view of (\ref{n_uple}), one just needs to compute the following integral:
\begin{multline}\label{to_integrate}
\int_{v=0}^\infty\!\!\int_{t=0}^\infty\!\! \int_{s>r>0}\!\!\int_{u=0}^w e^{-\delta
s}\PP[\underline{G}_\infty \in dr,\ -\underline{D}_\infty\in dv,\ \underrightarrow{\cal
D}_{L_b}-L_b\in dt,\ L_b\in ds,\\
\underrightarrow{D}_{L_b}-b\in du, b-D_{L_b-}\in dy,\ D_{L_b}-b\in dw]\\
= \rho(0) \int_{v=0}^\infty \!\! \int_{s>r>0}\hat{\cal U}(dr,dv)e^{-\delta s} {\cal
U}(ds-r,b+v-dy).\int_{u=0}^w \int_{t=0}^\infty \hat{\cal U}(dt,w-du).Q(dw+y),
\end{multline}
which we strive to do now. The first integral in the righthandside of (\ref{to_integrate})
verifies,
\begin{eqnarray}
&&\int_{v=0}^\infty \!\! \int_{s>r>0}\hat{\cal U}(dr,dv)e^{-\delta s} {\cal U}(ds-r,b+v-dy)\nonumber\\
&=&
\int_{v=0}^\infty \!\! \int_{r=0}^\infty \hat{\cal U}(dr,dv) \int_{s=r}^\infty e^{-\delta s} {\cal
U}(ds-r,b+v-dy)\nonumber\\
&=& \int_{v=0}^\infty \!\! \int_{r=0}^\infty \hat{\cal U}(dr,dv) e^{-\delta r}{\cal
U}_\delta(b+v-dy)\nonumber\\
&=& \int_{v=0}^\infty \!\! \int_{r=0}^\infty e^{-\delta r}\hat{\cal
U}(dr,dv)\left[W^{(\delta)'}(b+v-y)- \rho(\delta)W^{(\delta)}(b+v-y)\right] dy\quad \mbox{by Lemma
\ref{prop_calU_LT}}\nonumber\\
&=& \int_{v=0}^\infty \hat{\cal U}_\delta(dv)\left[
W^{(\delta)'}(b+v-y)-\rho(\delta)W^{(\delta)}(b+v-y)
\right] dy\nonumber\\
&=& \int_{v=0}^\infty   e^{-\rho(\delta)v}dv  \left[
W^{(\delta)'}(b+v-y)-\rho(\delta)W^{(\delta)}(b+v-y)\right] dy\quad \mbox{by
(\ref{calU_hat}).}\label{to_integrate3}
\end{eqnarray}
Relation (\ref{shifted}) yields that $e^{-\rho(\delta)v}
W^{(\delta)}(b-y+v)=e^{\rho(\delta)(b-y)}W_{\rho(\delta)}^{(0)}(b-y+v)$ which, from
(\ref{shifted2}), tends to $e^{\rho(\delta)(b-y)}\frac{1}{\varphi_D '(\rho(\delta)) }$ as
$v\to+\infty$. This justifies the following integration by parts:
\begin{multline}
\int_{v=0}^\infty  e^{-\rho(\delta)v} W^{(\delta)'}(b+v-y)dv= \left[ e^{-\rho(\delta)v}
W^{(\delta)}(b+v-y)\right]_{v=0}^\infty + \int_{v=0}^\infty  \rho(\delta)e^{-\rho(\delta)v}
W^{(\delta)}(b+v-y)dv\\
= e^{\rho(\delta)(b-y)}\frac{1}{\varphi_D '(\rho(\delta)) } - W^{(\delta)}(b-y)+ \int_{v=0}^\infty
\rho(\delta)e^{-\rho(\delta)v} W^{(\delta)}(b+v-y)dv, \label{IPP2}
\end{multline}
which, inserted in (\ref{to_integrate3}), yields the following simplification
\begin{equation}\label{to_integrate4}
\int_{v=0}^\infty \!\! \int_{s>r>0}\hat{\cal U}(dr,dv)e^{-\delta s} {\cal U}(ds-r,b+v-dy)=\left[
e^{\rho(\delta)(b-y)}\frac{1}{\varphi_D '(\rho(\delta)) }-W^{(\delta)}(b-y)\right]dy.
\end{equation}
The second integral in the righthandside of (\ref{to_integrate}) verifies
\begin{eqnarray}
\int_{u=0}^w \int_{t=0}^\infty \hat{\cal U}(dt,w-du)&=& \int_{u=0}^w \hat{\cal
U}_0(w-du)\nonumber\\
&=& \int_{u=0}^w e^{-\rho(0)(w-u)}du\quad \mbox{by (\ref{calU_hat})}\nonumber\\
&=& \frac{1}{\rho(0)}[1-e^{-\rho(0)w}].\label{to_integrate5}
\end{eqnarray}
Plugging (\ref{to_integrate4}) and (\ref{to_integrate5}) into (\ref{to_integrate}) yields
(\ref{Last_D_T_over_under}).
\end{proof}
\subsection{Examples}

We consider here some examples from those studied previously and for which last-passage time is relevant.

\paragraph{Brownian motion with positive drift} In the case where $G_t=\mu t$, $\mu>0$ and $D_t=G_t+\sigma
B_t=\mu t+\sigma B_t$, we have
\begin{eqnarray*}
\EE [D_1]&=&\mu,\\
W(a-b) & = & W^{(0)}(a-b)=\frac{2}{\mu} e^{-\mu (a-b)/\sigma^2}\sinh \left( \frac{a-b}{\sigma^2}
\mu\right) \quad \mbox{from (\ref{scale_BM})},\\
f_{D_t}(a)&=& \frac{1}{\sigma \sqrt{2\pi  t}} e^{-(a-\mu t)^2/(2\sigma^2 t)},
\end{eqnarray*}
which, plugged in (\ref{Last_D_T}) and (\ref{Last_D_T_density}), yields expression of the cdf
$t\mapsto \PP[L_b<t]$ as well as its cdf jointly to density of $D_t$. Note that by deriving this
expression of the cdf one obtains after some calculation the following density for $L_b$
$$
\PP[L_b\in dt]= \frac{\mu}{\sqrt{2\pi t}}e^{-\frac{(b-\mu t)^2}{2t}}dt,
$$
which agrees with the already known density of the last passage time of a Brownian motion with drift, see e.g. Expression (1.12) p.2 of \cite{PRY}.

\paragraph{Perturbed gamma process} In the case where $\{G_t,\ t\ge 0\}$ is a gamma process with shape parameter $\alpha$
and scale parameter $\xi$, densities of $G_t$ and $\sigma B_t$ are given by
$f_{G_t}(u)=\frac{u^{\alpha t -1}}{\Gamma(\alpha t)}\frac{e^{-u/\xi}}{\xi^{\alpha t}}$ and
$f_{\sigma B_t}(u)=\frac{1}{\sigma \sqrt{2\pi  t}} e^{-u^2/(2\sigma^2 t)}$. We also recall that function $H(\delta,x)$ defined in
Proposition \ref{expr_scale_fc} has expression given in (\ref{EDO_W}) with characteristics of the gamma perturbed process being given by (\ref{car_pert_gamma}). Hence a bit of calculation yields
\begin{eqnarray*}
\EE [D_1]&=&\alpha \xi,\\
W(a-b) &  =& \int_0^{a-b} e^{-\rho(\delta)(a-b-y)}H(\delta,y)dy,\\
f_{D_t}(a)&=& f_{G_t}\star f_{\sigma B_t}(a)= \frac{1}{\sigma \sqrt{2\pi  t} \Gamma(\alpha
t)\xi^{\alpha t} }\int_0^{+\infty} u^{\alpha t -1} e^{-u/\xi}
e^{-(a-u)^2/(2\sigma^2 t)} du,\\
&=& \frac{e^{-a^2/(2\sigma^2 t)}}{\sigma \sqrt{2\pi  t} \Gamma(\alpha t)\xi^{\alpha t}
}\int_0^{+\infty} u^{\alpha t -1}  e^{-\frac{1}{2\sigma^2 t}\left( u^2+\left( \frac{2\sigma^2
t}{\xi}-2a\right)u\right)}du\\
&=& \frac{e^{-a^2/(2\sigma^2 t)}}{\sigma \sqrt{2\pi  t} \Gamma(\alpha t)\xi^{\alpha t} } (\sigma
\sqrt{t})^{\alpha t -1} \int_0^{+\infty} x^{\alpha t -1}  e^{-\frac{1}{2} x^2-\frac{1}{2\sigma
\sqrt{t}}\left( \frac{2\sigma^2 t}{\xi}-2a\right)x}dx,\quad x:=u/(\sigma \sqrt{t}),\\
&=& \frac{(\sigma \sqrt{t})^{\alpha t -2}}{ \sqrt{2\pi } \Gamma(\alpha t)\xi^{\alpha t} }
e^{-\frac{a^2}{2\sigma^2 t}- \frac{1}{4\sigma^4}\left(\frac{\sigma^2 t}{\xi}-a\right)^4 }
\mbox{D}_{-\alpha t}\left( \frac{\sigma \sqrt{t}}{\xi}- \frac{a}{\sigma \sqrt{t}}\right)
\end{eqnarray*}
where $\Gamma(s)=\int_0^\infty e^{-t}t^{s-1}dt$, $s>0$, is the gamma function and
$\mbox{D}_p(z)=\frac{e^{-z^2/4}}{\Gamma(-p)} \int_0^\infty e^{-zx-x^2/2}x^{-p-1}dx$, $p<0$, is the
parabolic cylinder function (see (9.241.2) p.1064 of \cite{GR}). These expressions, plugged in
(\ref{Last_D_T_density}) and (\ref{Last_D_T_over_under}), yield expression of the cdf of $L_b$
jointly to density of $D_t$ as well as the Laplace transform of $L_b$ jointly to density of the
over and undershoot.


\paragraph{Perturbed compound Poisson process with phase-type distributed jumps} In the case where $\{G_t,\ t\ge 0\}$ is
a compound Poisson process with phase-type distributed jumps of parameters as in Section
\ref{subsection_examples}, we have, using same notations as in that section that density of shocks
is equal to $p(x)=\alpha e^{x{\bf T}}{\bf t}$ (see Theorem 1.5(b) p.218 of
\cite{Asmussen_ruinproba})and
\begin{eqnarray*}
\EE [D_1]&=& -\alpha {\bf T}^{-1} {\bf 1},\\
W(a-b) & = & \frac{2}{\sigma^2\varrho_0}\sum_{i\in {\cal I}_0}
A_{i,0}\frac{\xi_{i,0}}{\rho(0)+\xi_{i,0}}\left[e^{\rho(0)(a-b)} - e^{-\xi_{i,0} (a-b)}\right]\quad
\mbox{from (\ref{scale_PH}) with } \delta=0, \\
f_{D_t}(a)&=& \sum_{n=0}^\infty f_{\sigma B_t}\star p^{\star (n)}(a)e^{-\lambda t}\frac{(\lambda
t)^n}{n!}
\end{eqnarray*}
where $f_{\sigma B_t}(u)=\frac{1}{\sigma \sqrt{2\pi  t}} e^{-u^2/(2\sigma^2 t)}$. These
expressions, plugged in (\ref{Last_D_T_density})and (\ref{Last_D_T_over_under}), yield expression
of the cdf of $L_b$ jointly to density of $D_t$ as well as the Laplace transform of $L_b$ jointly
to density of the over and undershoot.

\subsection{Reflected processes}

As for the previous section dealing with first-passage time, we consider the last-passage time for the reflected version of perturbed increasing Lévy process.

\begin{thm}\label{thm_Last_D_T*_LT}
The Laplace transform of $L_b^*$ is given by
$$
\EE\left[e^{-\delta L_b^*}\right]=\EE[D_1]\int_b^\infty W'(a-b)\phi(\delta,a)da
$$
where we recall that $\phi(\delta,a)=\EE[e^{-\delta T_a}]=\phi_w(\delta,a)$ with $w\equiv 1$.
\end{thm}
\begin{proof}
We start similarly as in the proof of Theorem \ref{thm_Last_D_T} and let $T$ an independent r.v.
following an ${\cal E}(\delta)$. Event $\left[  L_b^* <T, D_T^*\in da\right]$ is equal to
$\left[D_T^*\in da,\ \{D^*_s\} \mbox{ will not hit level }b\mbox{ anymore after } T\right]$. Since
reflected process $\{D^*_t,\  t\ge 0\}$ behaves like the non reflected process $\{D_t,\  t\ge 0\}$
on event $\left[ \{D^*_s\} \mbox{ will not hit level }b\mbox{ anymore after } T\right]$ for $t\ge
T$, we have, for all $a>b$, and using the Markov property,
\begin{eqnarray}\label{proof_Last_D_T*_LT}
\PP\left[  L_b^* <T, D_T^*\in da\right]&=&\PP_{a-b} \left[  T_0=+\infty \right]\PP\left[ D^*_T \in
da\right]
\end{eqnarray}
where $\PP_{a-b} \left[  T_0=+\infty \right]$ is the probability that process $\{D_t,\  t\ge 0\}$
starting from $a-b$ will never hit $0$ and has expression $\EE[D_1]W(a-b)$, as observed in Proof of
Theorem \ref{thm_Last_D_T}. Since $W(z)=0$ on $z\le 0$, we have by Fubini theorem (and since $W(.)$
is a differentiable function by Remark \ref{scale_diff}),
\begin{eqnarray*}
\EE\left[e^{-\delta L_b^*}\right]=\int_{a=b}^\infty \PP\left[  L_b^* <T, D_T^*\in da\right]&=&
\EE[D_1]\int_{a=b}^\infty
W(a-b)\PP\left[ D^*_T \in da\right]\\
&=& \EE[D_1] \EE[W(D^*_T-b)]\\
&=& \EE[D_1] \EE\left[ \int_{a=b}^\infty W'(a-b)\nbu_{\{ D^*_T > a\}}da \right]\\
&=& \EE[D_1] \int_{a=b}^\infty W'(a-b) \PP[D^*_T > a]da.
\end{eqnarray*}
From Lemma \ref{duality_lemma}, we have that $\PP[D^*_T > a]=\PP[T_a \le T]$ which is equal to
$\phi(\delta,a)$, as $T$ follows an ${\cal E}(\delta)$ distribution. This yields the result.
\end{proof}

Again we emphasize that $\phi(\delta,a)=\EE[e^{-\delta T_a}]$ is available in practice either
through series (\ref{PK_phi}) in Proposition \ref{prop_renewal_LT}, or through (\ref{KP_T_b}) in
Proposition \ref{approach_KP}. Also note that proof of Theorem \ref{thm_Last_D_T*_LT} implicitly
yields the following side result.
\begin{prop}
Let $T$ be an independent ${\cal E}(\delta)$ distributed r.v. Then for all $a\ge b$ we have
\begin{equation}\label{eq_Last_D_T*_density}
\PP[L^*_b \ge T,\ D^*_T\in da]=-[1-E[D_1]W(a-b)]\frac{\partial}{\partial a}\phi(\delta,a)da.
\end{equation}
\end{prop}
\begin{proof}
As in showing (\ref{Last_D_T_density}), we use the fact that $\PP[L^*_b \ge T,\ D^*_T\in da]=\PP[D^*_T\in da]- \PP\left[  L_b^* <T, D_T^*\in da\right]$ as well as (\ref{proof_Last_D_T*_LT}) to derive that $\PP[L^*_b \ge T,\ D^*_T\in da]=[1-E[D_1]W(a-b)] \PP[D^*_T\in da]$. To obtain (\ref{eq_Last_D_T*_density}) we just need to prove that r.v. $D^*_T$ admits a density given by $\PP[D^*_T\in da]/da=-\frac{\partial}{\partial a}\phi(\delta,a)$. Indeed Lemma \ref{duality_lemma} yields that $\PP[D^*_T > a]=\PP[T_a \le T]=\EE[e^{-\delta T_a}]=\phi(\delta,a)$, thus what remains to prove is that $\EE[e^{-\delta T_a}]=\phi(\delta,a)$ is differentiable with respect to $a$. This can be seen thanks to the convenient expression (\ref{KP_T_b}) that yields that differentiability property since function $W^{(\delta)}$ is a differentiable function by Remark \ref{scale_diff} (and $Z^{(\delta)}$ is obviously differentiable by (\ref{eq:Zq})).
\end{proof}

\section{A maintenance policy}\label{Section_maintenance}

We now as an application consider the maintenance strategy described in Barker and Newby \cite{Barker_Newby}. Degradation of a certain component is modelled according to a process $\{X_t,\  t\ge 0\}$. We suppose that, without maintenance, $\{X_t,\  t\ge 0\}$ is a perturbed process with same parameters as $\{D_t,\  t\ge 0\}$and that failure occurs at the last passage time $L_b$ of level $b$ of the
degradation process. 
\\

Let us then consider the following maintenance rule. The component is inspected at times $(U_i)_{i=1,2,...}$ such that inter inspection time verifies $U_{i+1}-U_i=m(X_{U_i+})$, where $m(.)$ is some non increasing function. Let $d:\nbR \longrightarrow \nbR$ be some "maintenance function". On inspection at time $U_i$, one of the following actions is undertaken:
\begin{itemize}
\item either the system did not fail in interval $(U_{i-1},U_i]$, in which case preventive maintenance occurs and degradation process evolves like $\{D_t,\  t\ge 0\}$ with initial condition $D_0=d(x)$ up until time $U_{i+1}$, where $x$ is degradation state at instant $U_i-$; in other words one has $X_{U_i}=d(X_{U_{i}-})$,
\item or the system failed in interval $(U_{i-1},U_i]$ in which case it is repaired and degradation process starts anew, i.e. evolves like $\{D_t,\  t\ge 0\}$ with initial condition $D_0=0$.
\end{itemize}
We will suppose in this section that function $d(.)$ is differentiable from $\nbR$ to $\nbR$ and bijective. Note that these two assumptions are not too stringent and can be relaxed, in which case expressions of distributions computed in this section would only be more complicated.
\\

We then define r.v. $I$ as the first inspection after which system is reset, i.e.
$$
I=\inf\{ i \in \nbN|\ \mbox{failure occurred in }(U_{i-1},U_i]\}.
$$
This means that $T^*:=U_I$ is a regeneration time for the degradation process. Process $\{X_t,\  t\ge 0\}$ then behaves like independent copies of $\{D_t,\  t\ge 0\}$ in intervals $(U_i,U_{i+1}]$ with possibly different initial states. Figure \ref{trajec1} shows a sample path of $\{X_t,\  t\ge 0\}$, with failure in interval $(U_5,U_{6}]$and thus starting anew at time $U_6$ with $X_{U_6}=0$. Note that process $\{X_t,\  t\ge 0\}$ thus constructed is c\`adl\`ag and such that, given its state at any instant $U_k$, $\{X_t,\ t > U_k\}$ is independent from $\{X_t,\ t\in[0;U_k)\}$, i.e. from its history before $U_k$. This can be written as
$$
\left.\Big[ X_t,\ t \ge U_k\right|\ X_s,\ s\in [0,U_k]\Big] \stackrel{\cal D}{=} \left.\Big[ X_t,\ t \ge U_k\right|\ X_{U_k}\Big].
$$
\begin{figure}[!h]%
\begin{center}
\includegraphics[scale=0.6]{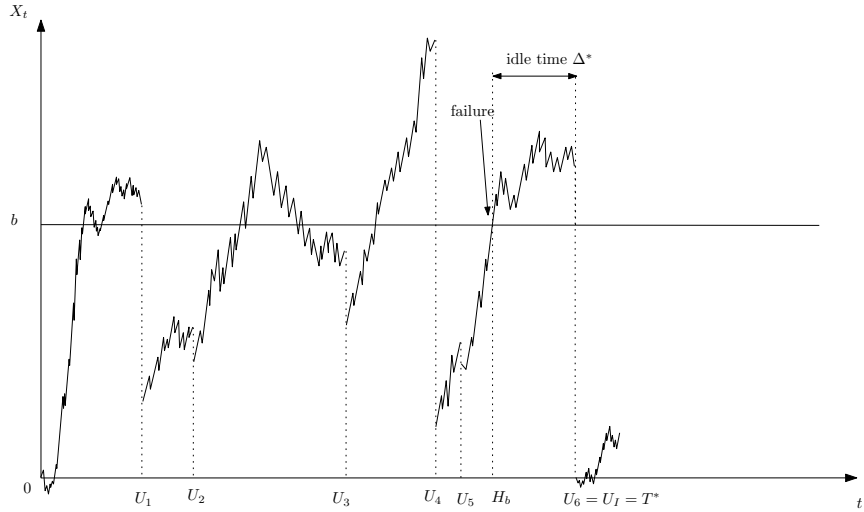}
\end{center}
\caption{Sample path of degradation process $\{X_t,\  t\ge 0\}$, with failure in $(U_5,U_{6}]$.}%
\label{trajec1}%
\end{figure}

We also introduce the idle time $\Delta^*$ which is the unavailability period of time during which component is down until next scheduled inspection:
$$
\Delta^*:=T^*-H_b\in [0,U_I-U_{I-1}]
$$
where $H_b$ is the failure time of the component and then necessarily lies in $[U_{I-1},U_I]$. We
are interested in quantities involving (possibly joined) distributions of $I$, $T^*$, $\Delta^*$ as
well as the state of the degradation process at inspection times. For this purpose we introduce the
following quantities:
\begin{itemize}
\item $A(x,dy):=\PP[L_b> m(x),\ d(D_{m(x)})\in dy|\ D_0=x]$ the distribution of the degradation process on inspection after maintenance jointly to the fact that there was no failure before inspection, given that degradation process starts at $x$,
\item $C(y):= \PP[L_b \le m(y)|\ D_0=y]$, the probability that failure occurred before next inspection, given that degradation process starts at $y$,
\item $C_r(y,z):=\PP[m(y)-L_b\ge z|\ m(y)\ge L_b,\ D_0=y]$, $z\in [0,m(y)]$, the survival function of the idle time given that degradation process starts at $y$.
\end{itemize}
These three quantities are easily obtained:
\begin{prop}\label{Prop1_maintenance}
We have the following expressions
\begin{eqnarray*}
A(x,dy)&=& [1-\EE[D_1]W(d^{-1}(y)-b+x)]\frac{f_{m(x)}(d^{-1}(y))}{d'[d^{-1}(y)]}dy,\\
C(y)&=& \int_{b-y}^\infty \EE[D_1]W(a-b+y)f_{m(y)}(a)da,\\
C_r(y,z)&=& \frac{1}{C(y)} \int_{b-y}^\infty \EE[D_1]W(a-b+y)f_{m(y)-z}(a)da.
\end{eqnarray*}
\end{prop}
\begin{proof}
We recall that we supposed that $d(.)$ is a one to one differentiable function out of practicality. Expression for $A(x,dy)$ simply comes from (\ref{Last_D_T_density}) with $t=m(x)$ and a simple change of variable $a=d^{-1}(y)$and remarking that last hitting time of level $b$ of process $\{D_t,\  t\ge 0\}$ with $D_0=x$ is the same in distribution as that of level $b-x$ of process $\{D_t,\  t\ge 0\}$ with $D_0=0$. Expression for $C(y)$ is obtained from (\ref{Last_D_T}) with $t=m(y)$ and $b:=b-y$ because of process starting from $y$. Finally expression for $C_r(y,z)$ comes from the fact that $$C_r(y,z)= \frac{\PP[m(y)-L_b\ge z|\ D_0=y]}{\PP[m(y)\ge L_b,|\ D_0=y]}= \frac{\PP[m(y)-L_b\ge z|\ D_0=y]}{C(y)}
$$
and using (\ref{Last_D_T}) with $T=m(y)-z$ and $b:=b-y$ to obtain expression of $\PP[m(y)-L_b\ge z|\ D_0=y]$.
\end{proof}

We may now state main results of this section that concern quantities of interest introduced at the
beginning of the section.
\begin{thm}\label{thm_maintenance}
Distribution of $I$ jointly to the state of the degradation process just after inspection and preventive maintenance is given by
\begin{equation}\label{thm1_maintenance}
\PP[I=i,\ X_{U_1}\in dy_1,...,\ X_{U_{i-1}}\in dy_{i-1}]=A(0,dy_1)\times A(y_1,dy_2)\times ...\times A(y_{i-2},dy_{i-1})\times C(y_{i-1}).
\end{equation}
Distribution of the idle time jointly to $I$ and the state of the degradation process just after inspection and preventive maintenance is given by
\begin{equation}\label{thm2_maintenance}
\PP[\Delta^*>z,\ I=i,\ X_{U_1+}\in dy_1,...,\ X_{U_{i-1}}\in dy_{i-1}]=A(0,dy_1)\times A(y_1,dy_2)\times ...\times A(y_{i-2},dy_{i-1})\times C_r(y_{i-1},z).
\end{equation}
\end{thm}
\begin{proof}
The first probability is obtained by writing it in the form $\PP\left[\cap_{k=1}^{i-1}E_k\cap F_i\right]$ where
\begin{eqnarray*}
E_k&=& \Big[\mbox{ no failure in } (U_{k-1};U_k],\ d(X_{U_k})\in dy_k\Big]\\
F_i&=& \Big[\mbox{ failure in } (U_{i-1};U_i]\Big].
\end{eqnarray*}
Since evolution of process $X_t$ in $t\in[U_i,U_{i+1})$ given $X_{U_i}$ is independent from $X_t$, $t\in[0,U_i)$, we may write that probability in the following form
$$
\PP[I=i,\ X_{U_1}\in dy_1,...,\ X_{U_{i-1}}\in dy_{i-1}]=\prod_{k=1}^{i-1}\PP[E_k|\
X_{U_{k-1}}=y_{k-1}]\times \PP[F_i|\ X_{U_{i-1}}=y_{i-1}]
$$
and conclude by the fact that by the stationary increment property we have $\PP[E_k|\ X_{U_{k-1}}=y_{k-1}]=A(y_{k-1},dy_k)$ and $\PP[F_i|\ X_{U_{i-1}}=y_{i-1}]= C(y_{i-1})$ in order to obtain (\ref{thm1_maintenance}). (\ref{thm2_maintenance}) is derived by similar arguments.
\end{proof}

Note that Theorem \ref{thm_maintenance} yields other interesting quantities. For example the expected time before reparation jointly to the number of inspections/maintenances is obtained thanks to (\ref{thm1_maintenance}) by
$$
\EE\left[T^*\nbu_{\{ I=i\}} \right]= \int_{(y_1,...,y_{i-1})\in \nbR^{i-1}} \left[ \sum_{k=1}^{i-1}f(y_k)\right]A(0,dy_1)\times A(y_1,dy_2)\times ...\times A(y_{i-2},dy_{i-1})\times C(y_{i-1}).
$$
\begin{remark}[Case of the reflected process]
It is possible to adapt the previous setting to the reflected process $\{D^*_t,\  t\ge 0\}$ and constructed a reflected degradation process $\{X^*_t,\  t\ge 0\}$ with inspection and maintenance by considering exponentially distributed inter-inspection times $U_{i+1}-U_i$ of which conditional distribution given $X_{U_i}$ is ${\cal E} (1/m(X_{U_i}))$, instead of deterministic times, where  $m(.)$ is the same function as in the non reflected caseand again featuring a maintenance function $d(.)$. Results from Theorem \ref{thm_Last_D_T*_LT} as well as equality (\ref{eq_Last_D_T*_density}) would yield similar expressions for $A(x,dy)$, $C(y)$ for exponentially distributed horizonand an equivalent of Theorem \ref{thm_maintenance} for such an inspection strategy could be obtained.
\end{remark}



\bibliographystyle{alpha}

\begin{thebibliography}{}

\end{thebibliography}


\begin{thebibliography}{10}

\bibitem{AH:survey}
M.S. Abdel-Hameed.
\newblock Degradation processes: an overview.
\newblock {\em Advances in degradation modelling. Applications to reliability, survival analysis and finance}, M.S. Nikulin \and N. Limnios \and N. Balakrishan \and W. Kahle \and C. Huber-Carol (Eds). Chapter 2:
17--25. Birkhaüser, 2010.

\bibitem{Asmussen_ruinproba}
S. Asmussen.
\newblock Ruin probabilities.
\newblock Advanced series on statistical sciences and applied probability, World Scientific, 2000.

\bibitem{Asmussen_Avram_Pistorius}
S. Asmussen, F. Avram and M. Pistorius.
\newblock Russian and American put options under exponential
phase-type L\'evy models.
\newblock {\em Stochastic Processes and their Applications}, 109(1): 79--111, 2004.

\bibitem{Barker_Newby}
C.T. Barker and M.J. Newby.
\newblock Optimal non-periodic inspection for a multivariate degradation model.
\newblock {\em Reliability Engineering and System Safety}, 94: 33--43, 2009.

\bibitem{Baurdoux}
E. Baurdoux.
\newblock Last exit before an exponential time for spectrally negative L\'evy processes.
\newblock {\em  Journal of Applied Probability}, 46: 542--558, 2009.

\bibitem{Bertoin}
J. Bertoin.
\newblock L\'evy processes.
\newblock Cambridge University Press, 2007.

\bibitem{Biffis_Morales}
E. Biffis and M. Morales.
\newblock On a generalization of the Gerber-Shiu function to path-dependent penalties.
\newblock {\em Insurance, Mathematics and Economics}, 46(1): 92--97, 2010.

\bibitem{BPS}
L. Bordes, C. Paroissin and A. Salami.
\newblock Parametric inference in a perturbed gamma degradation process.
\newblock Preprint, {\tt http://hal.archives-ouvertes.fr/hal-00535812/fr/}, 2010.

\bibitem{CKS}
T. Chan, A.E. Kyprianou and M. Savov.
\newblock Smoothness of scale functions for spectrally negative L\'evy processes.
\newblock {\em Probability Theory and Related Fields}, 2010.

\bibitem{CF77}
R.S. Chhikara and J.L. Folks.
\newblock The inverse Gaussian distribution as a lifetime model.
\newblock {\em Technometrics}, 19(4): 461--468, 1977.

\bibitem{Chiu_Yin}
S.N. Chiu and C. Yin.
\newblock Passage times for a spectrally negative L\'evy process with applications to risk theory.
\newblock {\em Bernoulli}, 11(3): 511--522, 2005.

\bibitem{Cox_Miller}
D.R. Cox and H.D. Miller.
\newblock The theory of stochastic processes.
\newblock Chapman and Hall/CRC, 1965.


\bibitem{Doney}
R.A. Doney.
\newblock Some excursion calculations for spectrally one-sided L\'evy processes.
\newblock {\em S\'eminaire de Probabilit\'es XXXVIII}, M. Emery, M. Ledoux and M. Yor (Eds), pp.5--15. Springer, 2005.


\bibitem{Egami}
M. Egami and K. Yamazaki.
\newblock On scale functions of spectrally negative L\'evy processes with phase-type jumps.
\newblock {\tt arXiv:1005.0064v3}, 2010.

\bibitem{FC}
J.L. Folks and R.S. Chhikara.
\newblock The inverse Gaussian distribution and its statistical application - A review.
\newblock {\em Journal of the Royal Statistical Society (B)}, 40: 263--275, 1978.

\bibitem{Garrido_Morales}
J. Garrido and M. Morales.
\newblock On the expected discounted penalty function for L\'evy risk processes.
\newblock {\em North American Actuarial Journal}, 10(4): 196--218, 2006.

\bibitem{GR}
I.S. Gradshteyn and I.M. Ryzhik.
\newblock Tables of integrals, seriesand products.
\newblock Academic Press, 1980.

\bibitem{Hubalek_Kyprianou}
F. Hubalek and A.E. Kyprianou.
\newblock Old and New Examples of Scale Functions for Spectrally Negative L\'evy Processes.
\newblock {\em Progress in Probability}, 63: 119--145, 2010.

\bibitem{Kyprianou_book}
A.E. Kyprianou.
\newblock Introductory lectures on fluctuations of L\'evy processes with applications.
\newblock Springer, 2006.

\bibitem{KP}
A.E. Kyprianou and Z. Palmowski.
\newblock A martingale review of some fluctuation theory for spectrally negative L\'evy processes.
\newblock {\em S\'eminaire de Probabilit\'es XXXVIII}, M. Emery, M. Ledoux and M. Yor (Eds), pp. 16--29. Springer, 2005.

\bibitem{KPR}
A.E. Kyprianou, J.C. Pardo and V. Rivero.
\newblock Exact and asymptotic $n$-tuple laws at first and last passage.
\newblock {\em Annals in Applied Probability}, 20(2):522--564, 2010.

\bibitem{Axel}
A. Lehmann.
\newblock Joint modelling of degradation and failure time data.
\newblock {\em Journal of Statistical Planning and Inference}, 5(1): 1693--1706, 2009.

\bibitem{ParkPadgett}
C. Park \and W.J. Padgett.
\newblock Accelerated degradation models for failure based on geometric Brownian motion and gamma
processes.
\newblock {\em Lifetime Data Analysis}, 11:511--527, 2005.

\bibitem{Pistorius}
M.R. Pistorius.
\newblock On exit and ergodicity of the spectrally negative L\'evy process reflected at its infinmum.
\newblock {\em Journal of Theoretical Probability}, 17(1): 183--220, 2004.

\bibitem{Pistorius2}
M.R. Pistorius.
\newblock  A potential-theoretical review of some exit problems of spectrally negative L\'evy processes.
\newblock {\em S\'eminaire de Probabilit\'es XXXVIII}, M. Emery, M. Ledoux and M. Yor (Eds), pp. 30--41. Springer, 2005.

\bibitem{PRY}
C. Profeta, B. Roynette and M. Yor.
\newblock Option prices as probabilities.
\newblock Springer-Finance, 2010.

\bibitem{RVV}
B. Roynette, P. Vallois and A. Volpi.
\newblock Asymptotic behavior of the hitting time, overshoot and undershoot for some L\'evy processes.
\newblock {\em ESAIM Probability and Statistics}, 12: 58--97, 2008.

\bibitem{ShakedShanki}
M. Shaked and J.G. Shanthikumar.
\newblock On the first-passage times of pure jump processes.
\newblock {\em  Journal of Applied Probability}, 25(3): 501--509, 1988.


\bibitem{TW}
C.C.L. Tsai and G.E. Willmot.
\newblock A generalized defective renewal equation for the surplus process perturbed by diffusion.
\newblock {\em Insurance, Mathematics and Economics}, 30: 51--66, 2002.

\end{thebibliography}

\end{document}